\documentclass{article}

\usepackage{fancyhdr}
\usepackage{graphicx}
\usepackage{amssymb}
\usepackage{amsmath}
\usepackage{amsfonts}
\usepackage{amsfonts}
\usepackage{amsthm}
\usepackage[german, english]{babel}

\theoremstyle{theorem}
\newtheorem{thm}{Theorem}[]
\newtheorem*{thm*}{Theorem}
\newtheorem{lem}[thm]{Lemma}

\normalsize

\begin{document}

\title       {Stable lengths on the pants graph are rational.}
\author   {Ingrid Irmer}
\maketitle

\begin{abstract} For the pants graph, there is little known about the behaviour of geodesics, as opposed to quasigeodesics.  Brock-Masur-Minsky showed that geodesics or geodesic segments connecting endpoints satisfying a bounded combinatorics condition, such as the stable/unstable laminations of a pseudo-Anosov, all have bounded combinatorics, \textit{outside of annuli}. In this paper it is shown that there exist geodesics that also have bounded combinatorics within annuli. These geodesics are shown to have finiteness properties analogous to those of tight geodesics in the complex of curves, from which rationality of stable lengths of pseudo-Anosovs acting on the pants graph then follows from the arguments of Bowditch for the curve complex.
\end{abstract}

\section{Introduction}
Suppose $S$ is a closed, orientable, connected surface with genus at least 2, and let $C_{p}(S)$ be the pants graph of $S$. The vertices of $C_{p}(S)$ represent pants decompositions of $S$, with edges between vertices representing pants decompositions whose component curves intersect minimally. A precise definition is given in section \ref{defns}.\\

Let $g$ be a pseudo-Anosov element of the mapping class group Mod($S$) of $S$, with unstable limit point $b$ and stable limit point $e$. The limit points $b$ and $e$ are laminations that satisfy the $K$-bounded combinatorics condition defined in \cite{BMM}, see section \ref{defns}. Since the pants graph is neither hyperbolic nor relatively hyperbolic relative to any nontrivial collection of subsets, \cite{BDM}, in general there is no geodesic stability. The $K$-bounded combinatorics condition is used in \cite{BMM} to make sense of a boundary at infinity.\\

The \textit{stable length} of $g$ is defined to be 
\begin{equation*}
\lim_{n\rightarrow \infty}\frac{d(v,g^{n}v)}{n}
\end{equation*}
where $d(*,*)$ is the usual combinatorial distance on $C_{P}(S)$ obtained by assigning each edge length one, and $v$ is any vertex of $C_{P}(S)$. Since the mapping class group acts by isometry, it is not hard to see that this quantity is locally constant, and hence independent of $v$, due to connectedness of $C_{P}(S)$.\\

The Nielsen-Thurston classification of mapping classes states that every mapping class is either pseudo-Anosov, periodic or reducible. These categories are known to have many properties in common with hyperbolic, elliptic and parabolic isometries, respectively, of the hyperbolic plane. One property of a hyperbolic isometry $h$ acting on the complex plane $\mathbb{H}^{2}$ is that the hyperbolic isometry leaves invariant a geodesic connecting its limit points at infinity; the \textit{axis} of $h$. Stable length of a pseudo-Anosov on $C_{P}(S)$ is shown to be rational by constructing a geodesic in $C_{P}(S)$ left invariant by some power of the pseudo-Anosov.\\

It was shown in \cite{MasurandMinskyI} that the stable length of a pseudo-Anosov acting on Harvey's complex of curves is nonzero. Since the stable length of a pseudo-Anosov acting on the pants graph is no less than its stable length acting on the complex of curves, it follows that this must also be nonzero. \\

In this paper it is shown that
\begin{thm}
\label{maintheorem}
Any pseudo-Anosov $g$ has rational stable length on $C_{P}(S)$.
\end{thm}
answering a question of Benson Farb, \cite{conversation}.\\


The major difficulty in working with curve complexes is that, apart from exceptional cases, they are not locally compact. In \cite{MasurandMinskyII} the notion of a tight geodesic was defined, in order to get around this problem. In \cite{Tight}, a slightly modified definition of tightness is given, which is the one given in section \ref{defns}.\\

To prove theorem \ref{maintheorem}, we would like an analogue of the notion of tight geodesics for the pants graph. This is defined in subsection \ref{greedy}. Essentially, these are the geodesics that pass through subsurfaces in the most convex way possible.\\


Let $M$ be the mapping torus of the pseudo-Anosov $g$, with its unique, finite volume hyperbolic metric (the existence of which follows from Mostov rigidity and \cite{Thurston}), and let $\tilde{M}\equiv S\times \mathbb{R}$ be the infinite cyclic covering space corresponding to the fiber. Fix an inclusion of $S$ into $\tilde{M}$, and identify curves on $S$ with curves in $\tilde{M}$. The \textit{length} of a curve is then the length of its geodesic representative in $\tilde{M}$. The finiteness properties of tight geodesics used in \cite{Bowditch} to prove rationality of stable length of a pseudo-Anosov acting on the complex of curves came from relating vertices of tight geodesics to short curves in $\tilde{M}$. The important property of these short curves being that there are only finitely many orbits of short curves under the action of $\langle g \rangle$.\\

Since edges only exist between pants decompositions that intersect minimally, increasing intersection number is related to increased distance in the pants graph, although in some complicated way. The basic idea of the proof of theorem \ref{maintheorem} is to use the fact that short curves minimise intersection number for a given curve complex distance. It follows that vertices on pants graph geodesics represent pants decompositions that tend to contain short curves, from which the finiteness properties underpinning theorem \ref{maintheorem} then follow from the arguments in \cite{Bowditch}.\\

\subsection{Acknowlegements} Thanks to Dan Margalit for telling me about this problem, and for his patience and helpful comments. Thanks also to Jason Behrstock for pointing out the reference \cite{BMM}, to ICERM's hospitality while writing this, and to Cyril Lecuire, Hugo Parlier and Saul Schleimer for helpful discussions. This work was funded by a MOE AcRF-Tier 2 WBS grant Number R-146-000-143-112.

\subsection{Some standard definitions relating to curve complexes}\label{defns}
A \textit{curve} is an isotopy class of maps of $S^{1}$ into the manifold in question; here either the surface $S$ or the 3-manifold $\tilde{M}$. A curve will often be confused with the image of a particular representative of the isotopy class. When the (sub)surface has punctures or nonempty boundary, it will also be assumed that the curve is not homotopic onto the boundary or into a puncture. All curves are assumed to be simple, where intersection numbers of curves in $\tilde{M}$ are defined by projecting onto the image of a $\pi_{1}$ injective embedding of the surface $S$ into $\tilde{M}$.\\


\textbf{Pants Graph}. The \textit{pants graph} $C_{P}(S)$ is the graph defined by Hatcher-Thurston with vertex set consisting of isotopy classes of pants decompositions of the surfaces $S$. Two vertices are connected by an edge if they represent pants decompositions that can be connected by a so-called \textit{elementary move}. An elementary move takes a curve $c_1$ in a pants decomposition $p_1$ and replaces it with another curve $c_2$, such that $(p_{1}\setminus c_{1})\cup c_{2}$ is a new pants decomposition, and $c_1$ and $c_2$ intersect minimally in the component of $S\setminus (p_{1}\setminus c_{1})$ into which they can both be isotoped.\\

A geodesic in the pants graph will be said to \textit{pass through a curve $c$}, or alternatively, the curve $c$ will be said to be \textit{on the geodesic} if there is a vertex of the geodesic representing a pants decomposition containing the curve $c$.\\

\textbf{Curve complexes.} A complex of curves will be defined for the surface $S$, and also, in order to define subsurface projections, curve complexes for subsurfaces of $S$. Suppose $S_{g,p}$ is a surface with genus $g$ and $p$ punctures. Except for the annulus, the set of curves on the subsurface $S_{g,p}$ define the vertices of Harvey's complex of curves, $C(S_{g,p})$. Whenever $3g+p>4$, two or more vertices of the curve complex $C(S_{g,p})$ span a simplex if the curves they represent can be realised disjointly. For all other subsurfaces of interest except the annulus, namely the four punctured sphere and the once punctured torus, $C(S_{g,p})$ has an edge connecting any pair of vertices representing curves that intersect minimally. For the annulus, the definition of $C(S_{g,p})$ needs to be approached differently, and is given in section 2.4 of \cite{MasurandMinskyII}.\\






Distances in curve complexes are defined by assigning each edge length one.\\

The notion of \textit{subsurface projection} was defined in \cite{MasurandMinskyII}, as a means of breaking down curve complex problems into simpler pieces. Let $Y$ be an incompressible, nonperipheral connected open subsurface of $S$. If $c$ is a curve the intersects $Y$, the \textit{subsurface projection of $c$ to $Y$} is defined to be a union of curves in $Y$ obtained by surgering each arc of $c \cap Y$ along the boundary of $Y$. Since there are some choices involved, in \cite{MasurandMinskyII} it is shown that subsurface projections are coarsely well defined. \\

The \textit{distance in the subsurface projection to $Y$} of the vertices $v_1$ and $v_2$ of $C(S)$, $d_{Y}(v_{1},v_{2})$, is zero when one or both of the vertices represent curves that can be isotoped out of $Y$. Otherwise, it is the distance in $C(Y)$ of the subsurface projections to $Y$ of the vertices $v_1$ and $v_2$, which is shown to be coarsely well defined. The subsurface projection to an annulus with core curve $c$ basically counts the number of times a curve is Dehn twisted around $c$ relative to another curve, and is quasi-isometric to $\mathbb{Z}$.\\

\textbf{Bounded Combinatorics}. Suppose $v_1$ and $v_2$ are vertices of $C(S)$ or stable/unstable laminations of a pseudo-Anosov. For any incompressible, connected open subsurface $Y$ of $S$, $d_{Y}(v_{1},v_{2})$ is defined as above. The pair $(v_{1},v_{2})$ is said to have $K$ \textit{bounded combinatorics}, \cite{BMM}, if there is an upper bound $K$ on the distance in the subsurface projection between $v_{1}$ and $v_{2}$ to all $Y$. The concept of $K$ bounded combinatorics is defined analogously for vertices of $C_{P}(S)$\\

A path in $C(S)$ or $C_{P}(S)$ will be said to be $K$ \textit{bounded} if this is true for all pairs of vertices through which it passes.\\



Let $\mathcal{F}(b,e)$ be the graph of geodesics in $C_{P}(S)$ connecting $b$ to $e$, and $\mathcal{F}_{K}(b,e)$ be the (possibly empty) subgraph of $K$ bounded geodesics.\\


\subsection{A locally finite subgraph of $\mathcal{F}(b,e)$}\label{greedy}
The aim is to find a locally finite subgraph of $\mathcal{F}(b,e)$ closed under the action of $g$. By \textit{locally finite} is meant that there are only finitely many geodesics in the subgraph connecting any two vertices. \\

Let $p_1$ and $p_2$ be two pants decompositions. Since any two pants decompositions represent vertices in $C_{P}(S)$ at finite distance, $p_1$ and $p_2$ have $K_1$ bounded combinatorics, for some $K_{1}>0$. Any two vertices along any pants geodesics connecting $p_1$ to $p_2$ necessarily have bounded conbinatorics, since the geodesics have finite length. A geodesic for which the supremum of $K$, taken over all pairs of vertices, is minimised, will be called a $K$ \textit{minimising geodesic}.\\

To see that there are only finitely many $K$ minimising geodesics connecting the vertices $v_1$ and $v_2$ corresponding to the pants decompositions $p_1$ and $p_2$, note that the $K$ bounded combinatorics condition restrains the number of twists that an elementary move can perform within a one holed torus or a four punctured sphere. Starting at the vertex $v_1$, there are finitely many edges emerging from $v_1$ that a $K$ bounded geodesic connecting $v_1$ to $v_2$ might take. From each of the endpoints of these finite number of edges, again there are only finitely many edges that a $K$ bounded geodesic might take, etc. The number of $K$ bounded geodesics is bounded from above by some constant $C$ raised to the power $d(v_{1},v_{2})$.\\

Since the action of the mapping class group on $C_{P}(S)$ preserves distance in subsurface projections, any mapping class maps $K$ minimising geodesic segments to $K$ minimising geodesic segments.\\

Suppose now that $p_1$ and $p_2$ are no longer pants decompositions, but a pair of laminations satisfying the $K$ bounded combinatorics condition, such as the limit points of a pseudo-Anosov. It is not yet clear that $K$ minimising geodesics connecting these two boundary points exist.\\

\subsection{Short curves and invariant quasigeodesics}
The pants graph is known to be quasi-isometric to Teichm\"uller space with the Weil-Petersson metric, \cite{Quasi}. Since $\tilde{M}$ has injectivity radius bounded from below, it follows from \cite{Minsky} that $\tilde{M}$ determines a family, $\mathcal{Q}_{2L}(b,e)$, of quasigeodesics in $C_{P}(S)$ connecting $b$ to $e$. The elements of $\mathcal{Q}_{2L}(b,e)$ only pass through short curves in $\tilde{M}$. In this context ``short'' means ``length on the order of magnitude of Bers' constant $L$'', for example, we make the arbitrary chioce, less than $2L$. Since $g$ acts by isometry, both on $C_{P}(S)$ and on $\tilde{M}$, $\mathcal{Q}_{2L}(b,e)$ is closed under the action of $g$. It then follows from exactly the same argument given in \cite{Bowditch}, that there is a quasigeodesic $Q$ in $\mathcal{Q}_{2L}(b,e)$ invariant under $g^m$ for some $m$. For the sake of completeness, this argument is sketched below.\\

\textbf{Bowditch argument.} Let $\mathcal{G}(b,e)$ be the graph of tight, directed geodesics in $C(S)$ connecting $b$ to $e$, and let $E$ be the set of directed edges of $\mathcal{G}(b,e)$. The mapping class group maps tight geodesics to tight geodesics, so $\mathcal{G}(b,e)$ is closed under the action of $g$. The set $E/\langle g \rangle$ is shown to be finite by using an argument from \cite{Minsky} to relate vertices of tight geodesics in $\mathcal{G}(b,e)$ to short curves in $\tilde{M}$, of which there are only finitely many modulo the action of $g$.\\

It may not be the case that all geodesics contained in $\mathcal{G}(b,e)$ are tight, so let $\mathcal{L}(b,e)$ be the set of all geodesics contained in $\mathcal{G}(b,e)$.\\

To obtain a geodesic invariant under some power of $g$, it remains to show that the finite number of elements of $E/\langle g \rangle$ can be connected up to form a geodesic in a way that is not completely random. This is shown using an argument attributed to Delzant: assign each of the finite elements of $E/\langle g \rangle$ a number. A geodesic $\gamma$ in $\mathcal{L}(b,e)$ will be said to be \textit{lexicographically least} for all vertices $v,w$ of $\gamma$ if the sequence of labels of directed edges in the segment of $\gamma$ connecting $v$ to $w$ is lexicographically least amongst all geodesic segments in $\mathcal{G}(b,e)$ connecting $v$ to $w$. Let $\mathcal{L}_{L}(b,e)$ be the subgraph of lexicographically least geodesics in $\mathcal{G}(b,e)$. It is shown that:

\begin{itemize}
\item{$\mathcal{L}_{L}(b,e)$ is nonempty}
\item{$\mathcal{L}_{L}(b,e)$ is closed under the action of $g$}
\item{$\mathcal{L}_{L}(b,e)$ contains finitely many elements}
\end{itemize}

Since there is a finite, nonempty, set of geodesics connecting $b$ to $e$, closed under the action of $g$, it follows that some finite power of $g$ has an axis $\Box$\\

\textbf{Defining twists.} Using the quasi-isometry between the pants graph and Teichm\"uller space, it is possible to define an approximate notion of the number of Dehn twists performed by an elementary move. Alternatively, since a geodesic $\gamma$ in $\mathcal{F}(b,e)$ is within a bounded distance of $Q$ (theorem 4.4 of \cite{BMM}), and the short curves in $\tilde{M}$ satisfy the $K$-bounded combinatorics condition, lengths of curves in $\tilde{M}$ could be used. This is the approach taken here. Let $s_{L,c}$ be the set of all curves in $\tilde{M}$ of length less than twice Bers' constant that intersect $c$. A curve will be said to be \textit{twisted around $c$ at least $n$ times} if it has distance at least $n$ from any curve in $s_{L,c}$ in the subsurface projection to the annulus with core curve $c$.\\

\textbf{Basic Problem.} Although the elements of $\mathcal{F}(b,e)$ fellow travel $Q$, since this is not the marking graph, what might happen is that every geodesic in $\mathcal{F}(b,e)$ passes through curves whose length approaches infinity. In this case, the geodesic segments connecting points of $Q$ to points of $\mathcal{F}(b,e)$ perform arbitrarily large numbers of twists, taking short curves to long curves. Since $g$ acts by isometry both on $\tilde{M}$ and $C_{P}(S)$, if $g$ is to have an axis, it is necessary to rule out the possibility that all geodesics in $\mathcal{F}(b,e)$ pass through arbitrarily long curves.\\

\subsection{Short curves in the pants complex}
We finally have all the ingredients to start the proof of theorem 1.

\begin{proof}

Suppose $\gamma$ is a geodesic in $\mathcal{F}(b,e)$ passing through the curve $c_n$,  where $c_n$ is long because it has been twisted around a curve $c$ at least $n$ times. Cut the geodesic $\gamma$ at the vertex $v$ to obtain two rays; one connecting $v$ to $e$, call it $r_1$, and the other connecting $b$ to $v$, $r_2$. 

\begin{lem}\label{lemma}
If $n$ is sufficiently large, i.e. greater than a constant depending on $g$, both $r_1$ and $r_2$ have to pass through $c$. 
\end{lem}


\begin{proof}
By assumption, every geodesic connecting $v$ to the nearest point(s) on $Q$ performs a large number of twists around the curve $c$. Clearly, this can not be the case for every vertex of $\gamma$, because the limit points $b$ and $e$ have $K$ bounded combinatorics. It follows that $\gamma$ necessarily passes through curves that are not twisted a large number of times around $c$ and have arbitrarily large intersection number with $c$.\\


Fact - Any curve that intersects $c_n$ minimally within a four punctured sphere or a once punctured torus either does not pass through the annulus with core curve $c$, or is twisted almost as many times as $c_n$. Similarly, if two curves are disjoint and both pass through the annulus with core curve $c$, the number of twists around $c$ can only differ by at most one.\\

Suppose $r_1$ does not pass through $c$. How does $r_1$ get from $c_n$ to any curve on $r_1$ that is not twisted a large number of times around $c$ and intersects $c$ more than twice?\\

Suppose $c$ intersects more than one curve in the pants decomposition corresponding to $v$. By the previous fact, in order to reach a pants decomposition that does not have a large subsurface projection to the annulus with core curve $c$, an elementary move at the vertex $v$ can not decrease the number of twists by more than one. An elementary move might increase/decrease by one the number of curves in the pants decomposition passing through the annulus with core curve $c$.\\

If $c$ only intersects one curve, i.e. $c_n$, in the pants decomposition corresponding to $v$, then $c$ and $c_n$ fill the subsurface in which elementary moves involving $c_n$ can occur. So an elementary move can undo the twists at most two at a time, or increase/decrease the intersection number of the pants decomposition with $c$, for example, via an elementary move that twists $c$ around $c_{n}$. However, it can not decrease the intersection number of the pants decomposition with $c$ to 0, because $c$ is the only nontrivial, nonperipheral curve in the 1-holed torus or 4-holed sphere in question disjoint from $c$.\\




For $n$ sufficiently large, an element of $\mathcal{F}(b,e)$ can not afford to undo the twists one or two at a time. By ``sufficiently large'', it suffices to take $n>N$. Here $N$ is twice the stable length of $g^k$, where $k$ is large enough to ensure that $g^k$ moves every curve at least curve complex distance $3$.\\

A symmetric argument shows that $r_2$ also passes through the curve $c$.\\
\end{proof}

\textbf{Existence of $K$ minimising geodesics.} Let $v_1$ be a vertex on $r_1$ containing $c$, and let $v_2$ be a vertex on $r_2$ containing $c$. Since the mapping class group acts on $C_{P}(S)$ by isometry, the mapping class $T_{c}^{-n}$, i.e. the inverse of $n$ Dehn twists around the simple curve $c$, takes a geodesic segment to a geodesic segment. The curves in the pants decompositions represented by $v_1$ and $v_2$ are all either disjoint from $c$ or $c$ itself, so both $v_1$ and $v_2$ are fixed by $T_{c}^{-n}$. Let $\gamma^{'}$ be the geodesic constructed by replacing the geodesic segment connecting $v_1$ and $v_2$ by its image under $T_{c}^{-n}$. Starting with a geodesic in $\mathcal{F}(b,e)$, by untwisting all the large subsurface projections in this way, a geodesic in $\mathcal{F}_{K}(b,e)$ is obtained for some $K$.\\

\textbf{Bounded combinatorics implies bounded length.} Now let $\gamma$ be a geodesic in $\mathcal{F}_{K}(b,e)$. From the argument in the previous paragraph, it can be assumed without loss of generality that a directed geodesic segment $\delta$ connecting a vertex on $\gamma$ to a nearest vertex on $Q$ does not perform any elementary moves that introduce more than $N$ twists. Let $l_{c}$ be an upper bound for the width of the collar of a curve on the Teichm\"uller quasigeodesic corresponding to $\tilde{M}$, and recall that $Q$ passes through curves whose length are all bounded from above by $2L$. Since $\gamma$ is contained within a bounded radius $D(\gamma)$ of $Q$, an upper bound on the length of curves can be obtained by iterating, as follows: for the first vertex on $\delta$, a bound of $b_{1}:=2L+4NL+2l_{c}=2(N+1)L+2l_{c}$ is obtained. For the nth vertex, $b_{n}=2(N+1)b_{n-1}+2l_{c}$.\\


The existence of an axis $\gamma_{a}$ invariant under some power $m$ of $g$ now follows directly from the argument in \cite{Bowditch}.\\

The stable length of $g$ is then equal to the rational number
\begin{equation*}
\frac{d(p, g^{m}p)}{m}
\end{equation*}
for any vertex $p$ on $\gamma_{a}$.

\end{proof}

\bibliographystyle{plain}
\bibliography{pantbib}

\begin{thebibliography}{10}

\bibitem{BDM}
J.~Behrstock, C.~Drutu, and L.~Moscher.
\newblock Thick metric spaces, relative hyperbolicity, and quasi-isometric
  rigidity.
\newblock {\em Mathematische Annalen}, 344:543.

\bibitem{Tight}
B.~Bowditch.
\newblock Length bounds on curves arising from tight geodesics.
\newblock {\em {G}eom. {F}unct. {A}nal.}, 17:1001, 2007.

\bibitem{Bowditch}
B.~Bowditch.
\newblock Tight geodesics in the curve complex.
\newblock {\em Invent. math.}, 171:281--300, 2008.

\bibitem{Quasi}
J.~Brock.
\newblock The {W}eil-{P}etersson metric and volumes of 3-dimensional hyperbolic
  convex cores.
\newblock {\em J. {A}mer. {M}ath. {S}oc.}, 16:495, 2003.

\bibitem{BMM}
{J}. {B}rock, {H}. {M}asur, and {Y}. {M}insky.
\newblock {A}symptotics of {W}eil-{P}etersson geodesics {II}: bounded geometry
  and unbounded entropy.
\newblock {\em {G}eom. and {F}unct. {A}nal.}, 21:820, 2011.

\bibitem{conversation}
D.~Margalit.
\newblock personal communication, 2013.

\bibitem{MasurandMinskyI}
H.~Masur and Y.~Minsky.
\newblock Geometry of the complex of curves {I}: Hyperbolicity.
\newblock {\em Invent. Math.}, 138:103--149, 1999.

\bibitem{MasurandMinskyII}
H.~Masur and Y.~Minsky.
\newblock Geometry of the complex of curves {II}: Hierarchical {S}tructure.
\newblock {\em Geometric and Functional Analysis}, 10, 2000.

\bibitem{Minsky}
{Y}. {M}insky.
\newblock The classification of {K}leinian surface groups, {I}: models and
  bounds.
\newblock {\em {A}nnals of {M}ath.}, page~1, 2010.

\bibitem{Thurston}
Thurston.
\newblock On the geometry and dynamics of diffeomorphisms of surfaces.
\newblock {\em {B}ulletin of the {A}merican {M}athematical {S}ociety}, 19:417,
  1988.

\end{thebibliography}
\end{document}